\documentclass[12pt,a4paper]{article}

\usepackage[latin1]{inputenc}

\usepackage{amsmath, relsize, enumerate}
\usepackage{amsfonts, amsthm, dsfont}
\usepackage{amssymb}
\usepackage{graphicx}
\usepackage{bm}

\def\a{{\alpha}}

\def\g{{\gamma}}
\def\d{{\delta}}
\def\D{{\Delta}}
\def\k{{\kappa}}
\def\l{{\lambda}}
\def\m{{\mu}}
\def\n{{\nu}}
\def\s{{\sigma}}
\def\S{{\Sigma}}

\def\1{{\mathds{1}}}
\def\C{{\mathds{C}}}
\def\E{{\mathds{E}}}

\def\P{{\mathds{P}}}
\def\Q{{\mathds{Q}}}
\def\R{{\mathds{R}}}

\def\bN{{\mathbf N}}

\def\dm{{\mathrm{d}}}

\def\Cov{{\C}{\rm ov}}
\def\Var{{\mathds V}}

\newcommand{\cC}{\mathcal{C}}
\newcommand{\cF}{\mathcal{F}}

\newcommand{\cP}{\mathcal{P}}
\newcommand{\cR}{\mathcal{R}}
\newcommand{\cS}{\mathcal{S}}
\newcommand{\cU}{\mathcal{U}}
\newcommand{\cV}{\mathcal{V}}

\newtheorem{theorem}{Theorem}[section]
\newtheorem{corollary}[theorem]{Corollary}

\theoremstyle{definition}
\newtheorem{definition}{Definition}[section]
\theoremstyle{remark}

\author{
Grace Akinwande
\thanks{Institut f\"ur Mathematik,
Universit\"at Osnabr\"uck, Germany, gakinwande@uni-osnabrueck.de}
\thanks{supported by the German Research Foundation DFG-GRK 1916}
\and 
Matthias Reitzner
\thanks{Institut f\"ur Mathematik,
Universit\"at Osnabr\"uck, Germany, matthias.reitzner@uni-osnabrueck.de}
}

\title{\textbf{Multivariate Central Limit Theorems for Random Simplicial Complexes}}
\date{}

\begin{document}
\maketitle

\section*{Abstract}
Consider a Poisson point process within a convex set in a Euclidean space. The Vietoris-Rips complex is the clique complex over the graph connecting all pairs of points with distance at most  $\d$. Summing powers of the volume of all $k$-dimensional faces defines the volume-power functionals of these random simplicial complexes. The asymptotic behavior of the volume-power functionals of the Vietoris-Rips complex is investigated as the intensity of the underlying Poisson point process tends to infinity and the distance parameter goes to zero. Univariate and multivariate central limit theorems are proven.
Analogous results for the \v{C}ech complex are given.

%{MSC}. Primary 60D05; Secondary 60F06, 55U10.

%{Keywords}. Random simplicial complex, Poisson point process, central limit theorem.

\section{Introduction}
Random simplicial complexes have been an extensively studied object within the last decade. In principle it started in the 60's when the random combinatorial Erd\"os-R\'enyi graphs and the random geometric Gilbert graph were defined by Erd\"os, R\'enyi and Gilbert. They obtained considerable interest and subgraph counts have been at the core of investigations from the very beginning. Clearly, clique counts are covered by these results. Yet a systematic study of the simplicial complexes built over random graphs mainly occurred in the last 20 years. 
In most cases the combinatorial structure of the complexes have been investigated, e.g. the $\bm f$-vector of the Vietoris-Rips complex over the Gilbert graph which is precisely the clique count, and only few results dealt with metric properties. 
In this paper we prove multivariate central limit theorems for volume-power functionals of the Vietoris-Rips and \v{C}ech complexes. As a special case we obtain a multivariate central limit theorem for the $\bm f$-vector of these complexes.

\bigskip
Assume that $W\subset \R^d$ is a compact convex set of volume $1$, and  that $\eta_t$ is a Poisson point process on $W \subset \R^d$ with constant intensity $t>0$. Let $\d_t$ be positive real numbers depending on $t$ such that $\d_t \to 0$ as $t \to \infty$.
The \emph{Vietoris-Rips complex} $\cR (\eta_t, \d_t)  $ takes the points of $\eta_t$ as its vertices, and a $(k+1)$-tuple of points $\{x_0, \dots, x_k\}$ in $\eta_t$ as $k$-simplex if all pairwise distances satisfy 
\begin{equation*}\label{def:dist}
\|x_i-x_j\|\leq\d_t .
\end{equation*}
The well known Gilbert graph is the one-skeleton of the Vietoris-Rips complex, and the Vietoris-Rips complex is just the clique complex of the Gilbert graph. For a thorough analysis of the Gilbert graph we refer to the seminal book by Penrose \cite{P03}. At this time, and with some of the main steps in this book, the modern limit theory for the Gilbert graph started. In particular, Penrose and his co-authors introduced Stein's method to this problem which yielded several important limit theorems.

Another natural simplicial complex over $\eta_t$ is the \emph{\v{C}ech complex} $\cC(\eta_t,\d_t)$,
whose $k$-simplices are all subsets
$\{x_0,\dotsc,x_k\}\subset \eta_t$ admitting a point $y\in \R^d$ with
$\|x_i-y\| \leq \tfrac{\d_t}{2} $ for all $0\leq i\leq k$.

The last years have seen significant developments on questions concerning these random simplicial complexes due to intrinsic mathematical interest and highly important applications. E.g., by the nerve lemma the combinatorics and topology of the \v{C}ech complex $\cC(\eta_t, \d_t)$  and the Boolean model from stochastic geometry, see \cite{SW08}, coincide and hence results on the \v{C}ech complex can be transferred to results on the Boolean model. Further, to study communication networks the Vietoris-Rips complex $\cR(\eta_t, \d_t)$ turned out to be a powerful model. And at last, in recent years \v{C}ech complexes and Vietoris-Rips complexes found highly interesting applications in topological data analysis. For an introduction to this we refer to the surveys on persistent homology by Carlsson \cite{Carl_2009} and Ghrist \cite{Ghrist_2008}. Though there is a large number of applications, to the best of our knowledge the multivariate covariance structure of the $\bm f$-vector of the random simplicial complexes have not been investigated.

\medskip
In this article we are interested in the limiting structure of
$ \cR(\eta_t, \d_t)$ and $\cC(\eta_t,\d_t) $ 
as the intensity $t$ is large and the distance $\delta_t$ is small. 
Denote by $\cF_k(\Delta)$ the set of $k$ faces of a simplicial complex $\Delta$, by $f_k(\Delta)=|\cF_k(\Delta)|$ the number of $k$-dimensional faces of the complex, and by $\bm f=(f_0, \dots, f_n)$ the $\bm f$-vector. For the Vietoris-Rips complex this is just the number of $k$-cliques of the Gilbert graph and thus is well investigated. E.g. expectation, variance, and central limit theorems are contained in Penrose book \cite{P03}. For the \v{C}ech complex expectation, variances and central limit theorems for $f_k(\cC(\eta_t, \d_t))$ are due to Decreusefond et al.~\cite{DFRV_14}. 

The aim of this paper is to analyze the metric behavior of the random simplicial complexes in more detail. 
We investigate the multivariate distributional properties of the volume-power functionals of the Vietoris-Rips complex, defined for $k \leq d$ by
\begin{equation}
\label{eq:eqn1}
\cV_k^{(\a)}:=
\frac{1}{(k+1)!}\sum_{F \in \cF_k(\cR(\eta_t, \d_t))}  \lambda_k(F)^{\a}  ,
\end{equation}
where $\l_k$ is the $k$-dimensional Lebesgue measure.
We do not need the restriction $k \leq d$ when $\a=0$, as then $\cV_k^{(0)}=f_k(\cR(\eta_t, \d_t))$ is equivalent to counting the complete graphs in the complex. This investigation started with work of Reitzner, Schulte and Th\"ale \cite{RST16}. There for the length-power functional of the Gilbert graph $\cV_1^{(\a)}$  expectation, covariance, and also several limit theorems were derived using the Malliavin-Stein method for Poisson functionals. 
In Theorem \ref{thm:expt} we prove that the expectation for $\cV_k^{(\a)}$ is asymptotically 
\begin{equation} \label{eq:Efk_kurz}
 \frac{\, \m_{k}^{(\a)}}{(k+1)!} \, t^{k+1}\d_t^{k(\a+d)} 
\end{equation}
with an explicitly given constant $\m_{k}^{(\a)}$, and determine in Theorem \ref{thm:covar} the covariance  $\Cov ( \cV_{k_1} ^{(\a_1)}, \cV_{k_2} ^{(\a_2)})$. In particular we investigate the rank of the asymptotic covariance matrix in Theorem \ref{thm:covar-matrix} and Corollary \ref{cor:rank}.
Combining this with  central limit theorems \cite{RS13, S13, S16} for Poisson U-statistics we  obtain our main results. Namely, in Theorem \ref{thm:uni-clt-V} we prove univariate central limit theorems for $\cV_k^{(\a)}$,
$$
d \bigg(\frac{{\cV}_k^{\a}-\E {\cV}_k^{\a}}{\sqrt{\Var\,{\cV}_k^{\a}}}, N\bigg)
\leq 
c_k  t^{- \frac 12} \max\{( t \d_t^d)^{- \frac k2 } , 1\} 
$$
where $N$ is a standard Gaussian random variable, and multivariate central limit theorems for a vector of suitable normalized volume-power functions $(\widehat{\cV}_{k_1}^{(\a_1)}, \dots, \widehat{\cV}_{k_n}^{(\a_n)})$, $k_1 \leq \dots \leq k_n$. The convergence to a multivariate normal distribution also has speed of convergence 
$$
t^{- \frac 12} \,  \max \{  (t \d_t^d)^{-\frac 12 k_n},1 \} ,
$$
see Theorem \ref{thm:mult-clt-V}. 
Note that both theorems cover the whole range where by \eqref{eq:Efk_kurz} the expectation of the number of $k$-simplices tends to infinity. 
In the case $\a=0$, Theorem \ref{thm:uni-clt-V} gives univariate central limit theorems for the number of $k$-facets and Theorem \ref{thm:mult-clt-V} multivariate central limit theorems for the $\bm f$-vector of $\cR(\eta_t, \d_t)$. 

For the Vietoris-Rips complex, a univariate central limit theorem is already known due to work of Penrose \cite{P03} although the central limit theorems there come without error term. In a recent paper by Lachi\'eze-Rey, Schulte and Yukich \cite{LSY19}, error terms for a univariate central limit theorem have been obtained  in the so-called thermodynamic regime and the dense regime as a consequence of a much more general theorem for stabilizing functionals. For the Wasserstein distance a univariate central limit theorem for $f_k(\cC(\eta_t, \d_t))$ of the \v{C}ech complex with error bounds is due to Decreusefond et al. \cite{DFRV_14}.
We are not aware of multivariate central limit theorems for the $\bm f$-vector in the literature.

A modified Gilbert graph could also be defined by replacing the Euclidean distance by some other quantity. For this there are several interesting recent papers, see e.g. \cite{LP_2013a, LP_2013b, LNS_2019}

The analogous question for the \v{C}ech complex is of the same interest. Since all the proofs and results for the Vietoris-Rips complex carry over to  the \v{C}ech complex, we only address this question at the end of this paper in Section \ref{sec:Cech}.

\section{Preliminaries}

We work in $\R^d$, $d \geq 1$, with Euclidean norm $\| \cdot \|$ and Lebesgue measure $\l_d(\cdot)$. A $d$-dimensional ball with center $x\in\R^d$ and radius $r>0$ is denoted by $B^d(x,r)$ and for a non-negative integer $j$, $\kappa_j$ stands for the volume of the $j$-dimensional unit ball $B^j=B^j(0,1)$. In the following we fix a convex compact set $W \subset \R^d$ of unit volume.

\medskip
We use the Landau notation. That is, for $g,h:\R\to\R$ we write $g=o(h)$ if $\lim\limits_{t\to\infty}|g(t)|/|h(t)|=0$, $g=O(h)$ if $\limsup\limits_{t\to\infty}|g(t)|/|h(t)| < \infty$ and $g=\Theta(h)$ if $g=O(h)$ and $h=O(g)$.

\subsection{Poisson point processes}

Our underlying probability space is $(\Omega,\cF,\P)$; expectation, variance and covariance of random variables $X$ and $Y$ with respect to $\P$ are denoted by $\E X$, $\Var X$ and $\Cov (X,Y)$, respectively. We also write $\1(\,\cdot\,)$ for an indicator function.

Let $\bN(W)$ be the space of finite simple counting measures on $W$. As usual we identify a counting measure $\eta$ with its support, which forms a finite subset of $W$, cf.\ \cite[Lemma 3.1.4]{SW08}. Thus for $\eta \in \bN(W)$, $\eta= \{ x_1, \dots , x_n\} $ and a Borel set $A \subset W$, $\eta(A)$ is the number of points of $\eta$ falling in $A$ and $\eta\cap A$ stands for the restricted point configuration  $\{ x_1, \dots , x_n\} \cap A$. 

Assume that $\eta_t$ is a Poisson point process on $W$ of intensity $t>0$. 
This is a random variable defined on the probability space $(\Omega,\cF,\P)$ with values in $\bN(W)$, such that for each Borel set $A \subset W$ the number of random points in $W$ is Poisson distributed with parameter $t \l_d(A)$, and such that the number of points in disjoint sets is independent. Alternatively, one can think of $\eta_t$ as a random set of $\eta_t(W)$ random points, which are independently placed within $W$ according to the uniform distribution.
We denote by $\eta_{t,\neq}^k$ the set of all $k$-tuples of distinct points of $\eta_t$.

The main tool in our investigations is the multivariate Mecke formula for Poisson point processes. In the special case we need in this paper it says that
\begin{equation}\label{eq:Mecke}
\E\sum_{(x_1,\ldots,x_k)\in\eta^k_{t,\neq}}  f(x_1,\ldots,x_k)
= t^k\int\limits_{W^k} \E f (x_1,\ldots,x_k)\, dx_1 \ldots dx_k \,,
\end{equation}
where $k\geq 1$ is a fixed integer, and $f:W^k\times \bN(W)\to\R$ is a non-negative measurable function, cf. \cite[Corollary 3.2.3]{SW08}.

A Poisson U-statistic $F=F(\eta_t)$ of order $k$ is a Poisson functional of the form
$$F=\sum_{(x_1,\ldots,x_k)\in\eta_{t,\neq}^{(k)}}f(x_1,\ldots,x_k)$$ with $k\in\mathbb{N}$ and $f:W^k\rightarrow\bar{\R}$. 
We will always assume that $F$ is in $L^1(\P)$ which by the Mecke formula is guaranteed if $f$ is in $L^1$.
For more information on Poisson point processes and the modern developments connected to the application of Malliavin calculus to Poisson point processes we refer to \cite{LP-book, PR16}.

\subsection{Random simplicial complexes}

Choose some parameter $\d_t>0$ which may depend on $t$. The Gilbert graph is defined as the random graph with vertex set $\eta_t $ and edges between two points $\{x_0, x_1\} \subset \eta_t$ if 
$ \| x_0-x_1\| \leq \d_t$.

The Gilbert graph is the one-skeleton of the \emph{Vietoris-Rips complex}, a random simplicial complex whose $k$-dimensional faces are the abstract simplices 
$\{ x_0, \dots, x_k\} \subset \eta_t$ iff $\| x_i-x_j \| \leq \d_t$. I.e. the Vietoris-Rips complex is the clique complex of the Gilbert graph. We denote the Vietoris-Rips complex by $\cR(\eta_t, \d_t)$ and the set of $k$ faces by
$ \cF_k(\cR (\eta_t, \d_t))  $.
For $k \leq d$ the faces have a geometric realization which is just the convex hull $[x_0, \dots, x_k] \subset W$ of $x_0, \dots, x_k$.

We denote by $\D_{s} [x_0, \dots x_k]$ the $k$-dimensional volume of the convex hull of the points $x_0, \dots, x_k$ if all edges have length at most $s$, and set $\D_{s} [x_0, \dots, x_k]=0$ otherwise. In the case $k >d$ there is no $k$-dimensional realization and thus in this case we just define $\D_{s} [x_0, \dots x_k]^0=1 $ if and only if all pairwise distances are bounded by $s$. Thus for all $k \geq 0$,
$$ 
F \in \cF_k (\cR (\eta_t, \d_t)) \Leftrightarrow 
\D_{\d_t}(F)^0=1  .
$$
Note that for $k \leq d$ we identify the abstract simplex $F=\{x_0, \dots, x_k\}  \in \cF_k(\cV(\eta_t, \d_t))$ with its geometric realization $[x_0, \dots, x_k]$.

\medskip
The quantity we are interested in this paper is the volume-power functional of the Vietoris-Rips complex, already defined in \eqref{eq:eqn1},
\begin{align}
\cV_k^{(\a)}
& := \nonumber
\frac{1}{(k+1)!}\sum_{F \in \cF_k(\cR(\eta_t, \d_t))}  \lambda_k(F)^{\a} 
\\ &= \label{def:Vka}
\frac{1}{(k+1)!}\sum_{(x_0, \dots, x_k) \in \eta_{t, \neq}^k} \Delta_{\d_t}[x_0, \dots, x_k]^{\a}  .
\end{align}
If $k \leq d$, this functional is in $L^1(\P)$ for $\a > -d$, see Theorem \ref{thm:expt}, and for $k>d$ we just consider the case $\a=0$.

\medskip
Closely connected to the Vietoris-Rips complex is the \emph{\v{C}ech complex} $\cC(\eta_t, \d_t)$. This is the random simplicial complex where an abstract $k$-simplex $\{x_{0},\ldots,x_{k}\}$ is in $\cC(\eta_t,\d_t)$ if $\cap_{i=0}^k B(x_{i},\d_t/2)\neq\emptyset$. All results of this paper can verbatim be formulated for the \v{C}ech complex $\cC(\eta_t, \d_t)$ instead of the Vietoris-Rips complex. This just changes some of the constants, see Section \ref{sec:Cech}.

\subsection{Moment matrices}

Let $X$ be a random variable, and let $\bm c=(c_1, \dots c_n) \in\R^n$ be chosen such that $m_{c_i+c_j}=\E X^{c_i+c_j} $ exists for $i,j=1, \dots, n$. Denote by $M_X (\bm c)$ the generalized moment matrix 
$$ M_X (\bm c) = (m_{c_i+c_j})_{i,j=1, \dots, n}. $$
The following theorem gives a criterion whether the generalized moment matrix is of full rank. The result should be \lq well known\rq , but we could not find it in the literature. 
\begin{theorem}\label{thm:rank-moment-matrix}
The generalized moment matrix $M_X(\bm c)$ is positive semidefinite. Moreover $M_X(\bm c)\, p = 0$ for some $ p \in \R^n$ implies
$$
X \in \Big\{ x \in \R  \colon \sum_{i=1}^n p_i x^{c_i}=0 \Big\}  \ \ a.s.
$$
If, in particular, $c_i \neq c_j$ for all $i \neq j$, and ${\rm supp} (X) $ contains an interval, then $M_X( \bm c)$ is of full rank.
\end{theorem}
\begin{proof}
The proof is a modification of similar results on moment matrices, see e.g. Laurent \cite{L_2009}. 
For $p \in \R^n$ we have 
$$
p^T M_X(\bm c) p =
\sum_{ij} p_i m_{c_i+c_j} p_j 
= \E \sum_{ij} p_i X^{c_i} X^{c_j}  p_j  
= \E (\sum_{i} p_i X^{c_i} )^2 \geq 0
$$
and thus $M_X(\bm c)$ is positive semidefinite. Further, if $M_X(\bm c) p=0$, then also
$$
p^T M_X(\bm c) p = \E (\sum_{i} p_i X^{c_i} )^2 = 0
. $$
Hence with probability one $X$ takes values in the root of the function $\sum p_i x^{c_i}$. 
If ${\rm supp }(X)$ contains an interval $I$ then $\sum p_i x^{c_i}=0$ for all $x \in I$. This is only possible if $p_i=0$ for $i=1, \dots n$, because the functions $x^{c_1}, \dots, x^{c_n}$ are independent for $c_i \neq c_j$.
\end{proof}

\section{Moments of random simplices}

\subsection{The expectation of $\cV^{(\a)}_k$}

Choose $k$ points $X_1, \dots, X_k$ independently and uniformly in the unit ball. To shorten our notation, we write in the following $\{x_l\}_{l=j}^k $ for the point set $\{ x_j, \dots, x_k\}$. 
The random points $\{X_l\}_{l=1}^k$ and the origin form a random $k$-simplex. For $k \leq d$ we denote by $\m_{k}^{(\a)} $ the moment of order $\a$ of its volume if all edges are bounded by one.
\begin{equation*}\label{def:Cka}
\m_{k}^{(\a)} = 
\k_d^k\  \E \D_{1}[0, \{ X_l\}_{l=1}^k]^\a
=
\int\limits_{(B^d)^k}\D_{1} [0,\{x_l\}_{l=1}^k]^{\a}\ \dm x_1\cdots\dm x_k
\end{equation*}
with $\m_0^ {(\a)}=1$. In the case $k >d$ the definition only applies to $\a=0$.
 Since the volume of a simplex is bounded by the product of the length of the generating edges, we have
$$
\m_{k}^{(\a)} 
\leq
\int\limits_{(B^d)^k}\prod^k_{i=1} \|x_i\| ^{\a}\ \dm x_1\cdots\dm x_k
=
%\left( \int\limits_{B^d} \|x_1\|^{\a}\ \dm x_1 \right)^k
\left( \frac {d \k_d }{\a+d} \right)^k
$$
and thus is finite for $\a > -d$.

\begin{theorem}
\label{thm:expt}
Assume $\a>-d$ for $k \leq d$ and $\a=0$ for $k>d$. Then we have
$$
\E\cV_k^{(\a)}  
 =
 \frac{\, \m_{k}^{(\a)}}{(k+1)!} \, t^{k+1}\d_t^{k(\a+d)} (1+O(\d_t )) 
% \frac{\, \m_{k}^{(\a)}}{(k+1)!} \, (t \d_t^d)^{k+1}\d_t^{k\a -d} (1+O(\d_t )) 
$$
where the implicit constant in $O(\d_t )$ only depends on $W$.
\end{theorem}
\begin{proof}
We apply the multivariate Mecke formula \eqref{eq:Mecke} to the definition \eqref{def:Vka} of the volume-power functionals,  and substitute $\d_t(x_i-x_k)$ for $x_i$, $i \neq k$, to get  
\begin{eqnarray*}
\E\cV_k^{(\a)}
&=&
\frac{t^{k+1}}{(k+1)!}\int\limits_{W^{k+1}} \D_{\d_t} [\{x_l\}_{l=0}^k]^{\a}\ \dm x_0\cdots\dm x_k
\\ &=&
\frac{\d_t^{k(\a+d)}t^{k+1}}{(k+1)!}\int\limits_W\ \int\limits_{(\d_t^{-1}(W-x_k) \cap B^d)^k} 
\D_{1} [0,\{x_l\}_{l=0}^{k-1}]^{\a}\ \dm x_0\cdots\dm x_k  
\end{eqnarray*}
where the condition $\|x_i\|\leq1$ has been taken into account in the range of integration $x_i \in B^d$. As an upper bound we have
\begin{eqnarray*}
\E\cV_k^{(\a)}
&\leq &
\frac{\d_t^{k(\a+d)}t^{k+1}}{(k+1)!}\int\limits_W \int\limits_{(B^d)^k} 
\D_{1} [0,\{x_l\}_{l=0}^{k-1}]^{\a}\ \dm x_0 \cdots\dm x_k  .
\end{eqnarray*}
This proves that the expectation is finite for $\a > -d$.
For an estimate from below we consider the inner parallel set of $W$, 
$W_{-\d_t}=\{x \colon B^d(x,\d_t)\subset W\}$. 
Observe that for $x_k \in W_{- \d_t}$ we have $\d_t^{-1} (W-x_k ) \cap B^d = B^d$.
\begin{eqnarray*}
\E\cV_k^{(\a)}
& \geq &
\frac{\d_t^{k(\a+d)}t^{k+1}}{(k+1)!}\int\limits_{W_{-\d_t}} \ \int\limits_{(B^d)^k} 
\D_{1} [0,\{x_l\}_{l=0}^{k-1}]^{\a}
 \dm x_0\cdots\dm x_k 
\\ &=&
\frac{\d_t^{k(\a+d)}t^{k+1}}{(k+1)!} \m_{k}^{(\a)}V(W_{-\d_t}) 
\end{eqnarray*}
We thus obtain
$$\frac{\d_t^{k(\a+d)}t^{k+1}}{(k+1)!}\m_{k}^{(\a)}V(W_{-\d_t})\leq\E\cV_k^{(\a)}\leq\frac{\d_t^{k(\a+d)}t^{k+1}}{(k+1)!}\m_{k}^{(\a)} .$$
The well known inequality 
\begin{equation}\label{eq:innpar} 
V(W_{-\d_t}) \geq 1-S(W)\d_t  % 1 -S(W)\d_t
\end{equation}
for convex sets of volume one gives the desired result .
\end{proof}

\subsection{The typical $k$-simplex}

Assume in this subsection that $\eta_t$ is a stationary Poisson point process in $\R^d$ of intensity $t>0$.
Fix some $k \in \{ 1, \dots, d\}$. To introduce the notion of a \textit{typical $k$-simplex} of $\cR(\eta_t, \d_t)$ we take the usual approach, see e.g. \cite[Chapter 4.1]{SW08}. We define for a simplex $S$ the center $c(S)$ as the center of the inball in the corresponding affine plane, and denote by $\cS_k^{(0)}$ the set of all centered $k$-simplices, $c(S)=0$, in $\R^d$. For $A \subset \cS^{(0)}_k$ and a convex set $W \subset \R^d$ with $V_d(W)=1$ we define
\begin{eqnarray*}
\Q (A) 
&=&
\frac 1{\g_{t}\, (k+1)!} \,  
\E \sum_{(x_0,\ldots,x_k)\in\eta_{t,\neq}^{(k+1)}}  \D_{\d_t} [\{x_l\}_{l=0}^k]^{0} 
\\ && \hskip1cm 
\1 \Big( c([\{x_l\}_{l=0}^k]) \in W,\, [\{x_l\}_{l=0}^k]-c([\{x_l\}_{l=0}^k]) \in A \Big) . 
\end{eqnarray*}
This is the suitable normalized mean number of simplices with center in $W$ and shape in $A$. We choose $\g_{t}$ in such a way that $\Q (\cS^{(0)}_k) =1$ which means that $\Q$ is a probability measure on $\cS_k^{(0)}$. Thus
\begin{eqnarray*}
\g_{t} 
&=&
\frac{1}{(k+1)!} \, \E  \sum_{(x_0,\ldots,x_k)\in\eta_{t,\neq}^{(k+1)}}  \D_{\d_t} [\{x_l\}_{l=0}^k]^{0}  
\1(c(\{x_l\}_{l=0}^k) \in W)  
\end{eqnarray*}
and it is well known that neither $\Q$ nor $\g_{t}$ depends on the choice of $W$, because $\eta_t$ is a stationary Poisson point process.
It is immediate that 
$$
\1(\{x_l\}_{l=0}^k \subset W)  
\leq \1(c([\{x_l\}_{l=0}^k]) \in W)   \leq 
\1(\{x_l\}_{l=0}^k \subset W_{\d_t}).  
$$
Hence using Theorem \ref{thm:expt} for $W$ and $W_{\d_t}$ we obtain 
$$ \g_t =   \frac{\, \m_{k}^{(0)}}{(k+1)!} \, t^{k+1}\d_t^{k d} (1+O(\d_t )) 
$$
and thus $\g_t$ equals asymptotically the expected number of $k$-simplices in $W$.

The \textit{typical $k$-simplex} of the Vietoris-Rips complex $\cR(\eta_t,\d_t)$ is a random simplex $S_{\rm typ}$ chosen according to $\Q$.
Now a straightforward application of Theorem \ref{thm:expt} yields the $\a$-moment of the volume of the typical cell.
\begin{corollary}
The $\a$-moment of the typical $k$-simplex of the Vietoris-Rips complex $\cR(\eta_t,\d_t)$ is 
$$
\E V_k(S_{\rm typ})^\a =  \frac{ \m_{k}^{(\a)}}{\m_{k}^{(0)}} \, \d_t^{\a k} (1+O(\d_t )) .
$$
\end{corollary}
This just reflects the fact that the typical cell is a simplex with edges bounded by $\d_t$ and thus - up to boundary effects - is given by the random variable $\D_{\d_t}$.

In a similar way we can compute the $j$-volume of the union of the $j$-faces of the typical $k$-simplex. Denote by $\cF_{j}(P)$ the $j$-dimensional faces of a polytope $P$. Because a $k$-simplex has $k+1 \choose j+1$ faces of dimension $j$, the Mecke formula  yields
\begin{align*}
\E V_j(\cF_j(S_{\rm typ})) 
&=
\frac{1}{\g_t (k+1)!} \, 
\E  \!\! \sum\limits_{\substack{ (x_0,\ldots,x_k)\in\eta_{t,\neq}^{(k+1)} \\ c([\{x_l\}_{l=0}^k]) \in W }}  \!\!
\D_{\d_t} [\{x_l\}_{l=0}^k]^{0} 
\sum\limits_{F \in \cF_j([\{x_l\}_{l=0}^k])} \D_{\d_t} [F]  
\\[1ex] &=
\frac{1}{\g_t (k+1)!} \,  { k+1 \choose j+1} 
t^{k+1}
\\&\hskip1.5cm 
\int\limits_{c([\{x_l\}_{l=0}^k]) \in W }  \D_{\d_t} [\{x_l\}_{l=0}^k]^{0} 
\D_{\d_t} [\{x_l\}_{l=0}^j]  dx_0 \dots dx_k .
\end{align*}
We define
\begin{align}
\m_{j, k:j+1}^{(1, 0)}
=& \label{def:Cjk10}
\int\limits_{(B^d)^{k}}
\D_{1} [0,\{x_l\}_{l=1}^{j}]
\D_{1} [0, \{x_l\}_{l=1}^{k}]^{0}
\dm x_1\cdots \dm x_{k} .
\end{align}
\iffalse

\begin{align*}
\frac{1}{(k+1)!} 
&\, 
\E  \sum\limits_{\substack{ (x_0,\ldots,x_k)\in\eta_{t,\neq}^{(k+1)} \\ c([x_0, \dots, x_k]) \in W }}  \D_{\d_t} [x_0, \dots, x_k]^{0} 
\sum\limits_{F \in \cF_j([x_{0}, \dots , x_{k}])} \D_{\d_t} [F]  
\\ = &
\frac{1}{(k+1)!} \, { k+1 \choose j+1} 
\E  \sum\limits_{\substack{ (x_0,\ldots,x_k)\in\eta_{t,\neq}^{(k+1)} \\ c([x_0, \dots, x_k]) \in W }}  \D_{\d_t} [x_0, \dots, x_k]^{0} 
\D_{\d_t} [x_{0}, \dots , x_{j}]  
\\ = &
\frac{1}{(j+1)!(k-j)!}  t^{k+1}
\int\limits_{c([x_0, \dots, x_k]) \in W }  \D_{\d_t} [x_0, \dots, x_k]^{0} 
\D_{\d_t} [x_{0}, \dots , x_{j}]  dx_0 \dots dx_k
\\ \approx &
\frac{1}{(j+1)!(k-j)!}  t^{k+1}
\int\limits_{x_0, \dots, x_k \in W }  \D_{\d_t} [x_0, \dots, x_k]^{0} 
\D_{\d_t} [x_{0}, \dots , x_{j}]  dx_0 \dots dx_k
\\ = &
\frac{1}{(j+1)!(k-j)!}  t (t \d_t^d)^k \d_t^j 
\int\limits_{x_1,\dots, x_k \in B^d }   \D_{1} [0,x_1, \dots, x_k]^{0} 
\D_{1} [0, x_1, \dots , x_{j}]  dx_1 \dots dx_k
\\ = &
\frac{1}{(j+1)!(k-j)!}  t (t \d_t^d)^k \d_t^j 
\m^{(0,1)}_{k,j:j+1}
\end{align*}
\fi
Using a similar approach as in Theorem \ref{thm:expt} we obtain

\begin{align*}
\E V_j(\cF_j(S_{\rm typ})) 
&=
{k+1 \choose j+1} 
\frac{\m^{(1,0)}_{j,k:j+1} }{ \m_{k}^{(0)}} (1+O(\d_t ))  .
\end{align*}

\subsection{The Covariance Structure of $\cV^{(\a)}_k$}

We consider the covariance in two variables - in $\a$ and in $k$.
In the following we denote by $\m_{k_1, k_2:m}^{(\a_1, \a_2)} $ the mixed moment of the volume of two random simplices in $B^d$ with edge lengths bounded by one. For this  we put one vertex in the origin, choose in $B^d$ for the first simplex $k_1-m+1$ independent vertices $\{ X_1, \dots, X_{k_1-m+1}\}$ and $m-1$ further random vertices $\{ X_{k_1-m+2}, \dots, X_{k_1} \}$ defining $\D_{1} [0,\{X_l\}_{l=1}^{k_1}]$, with $1\leq m \leq \min\{ k_1, k_2\}+1$. We use the last $m-1$ vertices also for the second simplex and in addition the random vertices $\{X_{k_1+1} \dots X_{k_1+k_2-m+1} \}$ to define the second simplex for $\D_{1} [0, \{X_l\}_{l=k_1-m+2}^{k_1+k_2-m+1}]$. Assuming here w.l.o.g.  $k_1 \leq k_2$, we generalize the definition \eqref{def:Cjk10} to 
\begin{align*}\label{def:Caakkm}
\m_{k_1, k_2:m}^{(\a_1, \a_2)}
=& 
\int\limits_{(B^d)^{k_1+k_2+1-m}}
\D_{1} [0,\{x_l\}_{l=1}^{k_1}]^{\a_1} 
\D_{1} [0, \{x_l\}_{l=k_1-m+2}^{k_1+k_2-m+1}]^{\a_2}
\\ & \hskip6cm  \nonumber
\dm x_1\cdots \dm x_{k_1+k_2-m+1} 
\end{align*}
and set by symmetry $\m_{k_1, k_2:m}^{(\a_1, \a_2)} = \m_{k_2, k_1:m}^{(\a_2, \a_1)} $. 
Clearly for $k_1=0$ we have $\m_{0, k_2:m}^{(\a_1, \a_2)} = \m_{k_2}^{(\a_2)} $, and for $k_i>d$ we only allow $\a_i=0$.
Again, since the volume of a simplex is bounded by the product of the length of the generating edges, we have
$$
\m_{k_1, k_2:m}^{(\a_1, \a_2)} 
\leq
 \left( \frac {d \k_d }{\a_1+d} \right)^{k_1-m+1} \left( \frac {d \k_d }{\a_1+\a_2+d} \right)^{m-1}  \left( \frac {d \k_d }{\a_2+d} \right)^{k_2-m+1}
$$
for $k_1 \leq k_2$, and thus $\m_{k_1, k_2}^{(\a_1, \a_2)} $ is finite for $\min\{\a_1, \a_2, \a_1+\a_2\} > -d$.

In the following theorem we assume $k_1 \leq k_2$, $\a_i=0$ for $k_i>d$, and  $\min\{\a_1, \a_2, \a_1+\a_2\} > -d$.
\begin{theorem}\label{thm:covar}
The covariance is given by
\begin{align*}
\Cov\big(\cV_{k_1}^{(\a_1)},\cV_{k_2}^{(\a_2)}\big)
= &
\sum_{m=1}^{\min k_i+1} \frac{ \m_{k_1, k_2:m}^{(\a_1, \a_2)}
 }{m! \prod (k_i-m+1)!}
t^{\sum k_i \, -m+2} \d_t^{\sum (d+\a_i) k_i \,  -d(m-1) } 
\\ & \hskip8cm (1+o(1)) .
\end{align*}
In particular, for $2 \a > - d$ we have 
\begin{eqnarray*}
\Var\,\cV_{k}^{(\a)}
&=&
\sum_{m=1}^{k+1} \frac{ \m_{k, k:m}^{(\a, \a)}
 }{m! ((k-m+1)!)^2}
t^{2k \, -m+2} \d_t^{2(d+\a) k \,  -d(m-1) } 
(1+o(1)) .
\end{eqnarray*}
\end{theorem}

\begin{proof}
Without loss of generality, we again assume $k_1\leq k_2$. By definition,
\begin{align*}
\cV_{k_1}^{(\a_1)}\cV_{k_2}^{(\a_2)}
=
\frac{1}{\prod\limits_{i=1,2}(k_i+1)!}
&
\sum_{\substack{(x_0,\ldots,x_{k_1})\in\eta_{t,\neq}^{(k_1+1)} \\ (x'_0,\ldots,x'_{k_2})\in\eta_{t,\neq}^{(k_2+1)}}} 
\D_{\d_t} [\{x_l\}_{l=0}^{k_1}]^{\a_1}
\D_{\d_t} [\{x'_l\}_{l=0}^{k_2}]^{\a_2}
\end{align*}
Here, $m$ points of the $k_1$-tuple and $k_2$-tuple may coincide, $m=0,\ldots k_1+1$. We assume that $x_{k_1-m+1}=x'_{k_1-m+1}, \dots, x_{k_1}=x'_{k_1}$, multiply by ${ k_1 +1 \choose m} \frac{(k_2+1)!}{(k_2-m+1)!}$, and rename the variables $(x'_0, \dots, x'_{k_1-m})$ by $ (x_{k_2+1}, \dots, x_{k_1+k_2-m+1})$ . This yields
\begin{align*}
\cV_{k_1}^{(\a_1)}\cV_{k_2}^{(\a_2)}
= &
\sum\limits_{m=0}^{k_1+1}
\frac{1}{m! \prod\limits_{i=1,2}(k_i-m+1)!}
\\ & \hskip0.5cm 
\sum_{(x_{0},\ldots,x_{k_1+k_2-m+1})\in\eta_{t,\neq}^{(k_1+k_2-m+2)}}
\D_{\d_t} [\{x_l\}_{l=0}^{k_1}]^{\a_1} 
  \D_{\d_t}[\{x_l\}_{l=k_1-m+1}^{k_1+k_2-m+1}]^{\a_2}
\end{align*}
and applying the Mecke formula gives
\begin{align*}
\E\cV_{k_1}^{(\a)}\cV_{k_2}^{(\a)}
=&
\sum_{m=0}^{k_1+1} \frac{t^{k_1+k_2-m+2}}{m! \prod\limits_{i=1,2}(k_i-m+1)!}
\\ & 
\int\limits_{W^{k_1+k_2-m+2}}
\D_{\d_t} [\{x_l\}_{l=0}^{k_1}]^{\a_1} 
\D_{\d_t} [\{x_l\}_{l=k_1-m+1}^{k_1+k_2-m+1}]^{\a_2}
\\ & \hskip6cm 
\dm x_0\cdots\dm x_{k_1+k_2-m+1}.
\end{align*}
The first term of this sum with $m=0$ equals $\E\cV_{k_1}^{(\a)}\E\cV_{k_2}^{(\a)}$ and thus the covariance is given by the summands from $m=1$ to $m=k_1+1$. To obtain the asymptotic behavior of the covariance we follow the same approach as in the proof of Theorem~\ref{thm:expt}. We substitute $\d_t(x_i-x_{k_1})$ for $x_i$, $i \neq k_1$ to get  
\begin{align*}
\Cov\big(\cV_{k_1}^{(\a_1)}, & \cV_{k_2}^{(\a_2)}\big)
=
\sum_{m=1}^{k_1+1} \frac{t^{k_1+k_2-m+2} \d_t^{d(k_1+k_2-m+1)+ \a_1 k_1+\a_2 k_2} }{m! \prod\limits_{i=1,2}(k_i-m+1)!}
\\ & 
\int\limits_W 
\int\limits_{(\d_t^{-1}(W -x_{k_1}) \cap B^d)^{k_1+k_2+1-m}}
\D_{1} [0,\{x_l\}_{l=0}^{k_1-1}]^{\a_1} 
\D_1 [0, \{x_l\}_{\substack{l=k_1-m+1\\l \neq k_1}}^{k_1+k_2-m+1}]^{\a_2}
\\ & \hskip4.5cm 
\dm x_0\cdots\dm x_{k_1-1}\, \dm x_{k_1+1}\cdots\dm x_{k_1+k_2-m+1} \, dx_{k_1}.
\end{align*}
For an upper bound we use $\d_t^{-1}(W -x_{k_1}) \cap B^d \subset B^d$ and obtain
\begin{align*}
\Cov\big(\cV_{k_1}^{(\a_1)},\cV_{k_2}^{(\a_2)}\big)
\leq &
\sum_{m=1}^{k_1+1} \frac{t^{k_1+k_2-m+2} \d_t^{d(k_1+k_2-m+1)+ \a_1 k_1+\a_2 k_2} }{m! \prod\limits_{i=1,2}(k_i-m+1)!}
\m_{k_1, k_2:m}^{(\a_1, \a_2)} .
\end{align*}
Once again, for the lower bound we consider $x_{k_1}\in W_{-\d_t}$, which yields
\begin{align*}
\Cov\big(\cV_{k_1}^{(\a_1)},\cV_{k_2}^{(\a_2)}\big)
\geq &
\sum_{m=1}^{k_1+1} \frac{t^{k_1+k_2-m+2} \d_t^{d(k_1+k_2-m+1)+ \a_1 k_1+\a_2 k_2} }{m! \prod\limits_{i=1,2}(k_i-m+1)!}
V(W_{-\d_t}) \m_{k_1, k_2:m}^{(\a_1, \a_2)}
\end{align*}
and using the estimate \eqref{eq:innpar} for $V(W_{-\d_t})$ proves Theorem \ref{thm:covar}.
\end{proof}

It turns out to be useful to distinguish the behavior of the covariance in the different \emph{regimes} already introduced by Penrose \cite{P03}. We say that we are in a sparse regime if $\lim t \d_t^d =0$ for $t \to \infty$, in the thermodynamic regime if $\lim t \d_t^d \in (0, \infty)$, and in the dense regime if $\lim t \d_t ^d= \infty$.
We assume again $k_1 \leq k_2$, $\a_i=0$ for $k_i>d$, and  $\min\{\a_1, \a_2, \a_1+\a_2\} > -d$.
\begin{theorem}\label{thm:covar-normalized}
Set 
\iffalse
\begin{equation}\label{def:Qi-alt}
Q_i=\max\limits_{1\leq m\leq k_i+1}\{t^{k_i -\frac m2+1} \d_t^{(d+\a_i) k_i \,  -\frac d2 (m-1) } \} . 
\end{equation}
Much better:
\fi
\begin{equation}\label{def:Qi}
Q_i=
t^{\frac 12} \d_t^{ \a_i  k_i \,  } \max\limits_{1\leq m\leq k_i+1}\{(t \d_t^d)^{k_i-\frac {m-1}2} \} 
. 
\end{equation}
Define the normalized volume-power functional  by $\widehat{\cV}_{k_i}^{(\a_i)}= \cV_{k_i}^{(\a_i)}/Q_{i}$.
\begin{enumerate}[(i)]
\item 
In the sparse regime, where $\lim\limits_{t\rightarrow\infty}t\d_t^d=0$, we have
$\lim_{t \to \infty} \Cov\big(\widehat{\cV}_{k_1}^{(\a_1)},\widehat{\cV}_{k_2}^{(\a_2)}\big) =0
$
for $k_1 < k_2$, and for $k_1 = k_2=k$
\begin{eqnarray*}
\lim_{t \to \infty} \Cov\big(\widehat{\cV}_{k}^{(\a_1)},\widehat{\cV}_{k}^{(\a_2)}\big)
&=& 
\frac{  \m_{k}^{(\a_1 + \a_2)} }{(k+1)!} .
\end{eqnarray*}
\item 
In the dense regime, where $\lim\limits_{t\rightarrow\infty}t\d_t^d=\infty$, we have 
\begin{eqnarray}\label{eq:cov-ai-dense1}
\lim_{t \to \infty}  \Cov\big(\widehat{\cV}_{k_1}^{(\a_1)},\widehat{\cV}_{k_2}^{(\a_2)}\big)
&=& 
\frac{\m_{k_1}^{(\a_1)}}{k_1! }  \frac{\m_{k_2}^{(\a_2)}}{k_2!}  .
\end{eqnarray}
\item 
In the thermodynamic regime, where $\lim\limits_{t\rightarrow\infty}t\d_t^d=c\in(0,\infty)$, we have
\begin{eqnarray*}
\lim_{t \to \infty}
\Cov\big(\widehat{\cV}_{k_1}^{(\a_1)},\widehat{\cV}_{k_2}^{(\a_2)}\big)
&=&
\left\{\begin{array}{ll}
\sum\limits_{m=0}^{k_1} \frac{ \m_{k_1, k_2:k_1-m+1}^{(\a_1, \a_2)} }{(k_1-m+1)! m! (k_2-k_1+m)!} c^{\frac{k_2-k_1}2 +m},
&c \leq 1
\\[8pt]
\sum\limits_{m=0}^{ k_1} \frac{ \m_{k_1, k_2:m+1}^{(\a_1, \a_2)}
}{(m+1)! (k_1-m)!(k_2-m)!} c^{-m} ,
&c\geq 1.
\end{array} \right.
\end{eqnarray*}

$$
$$

\end{enumerate}
\end{theorem}

\begin{proof}

Recall that $k_1 \leq k_2$. Because 
\begin{equation}\label{eq:Ql}
Q_i= t^{\frac 12} \d_t^{\a_i k_i}
\max \{  (t \d_t^d)^{ k_i}  ,
(t \d_t^d)^{\frac 12 k_i }  \}  
\end{equation}
the behavior in the sparse and dense regimes are immediate. 

In the sparse regime $t \d_t^d \to 0$, and hence as soon as $t \d_t^d <1$, the maximum is attained for $m=k_1+1$, resp. $m=k_2+1$. Thus 
$$
Q_1 Q_2
=
t \, \d_t^{ \sum \a_i  k_i   }  (t \d_t^d)^{\sum \frac {k_i}2}  
$$ 
and 
\begin{eqnarray*}
\Cov\big(\widehat{\cV}_{k_1}^{(\a_1)},\widehat{\cV}_{k_2}^{(\a_2)}\big)
&=&
\frac{1}{Q_1Q_2} \Cov\big(\cV_{k_1}^{(\a_1)},\cV_{k_2}^{(\a_2)}\big)
\\ &=& \nonumber
 \frac{ \m_{k_1, k_2:k_1+1}^{(\a_1, \a_2)}
 }{(k_1+1)! (k_2-k_1)!} (t\d_t^d)^{\frac {k_2-k_1}2 } 
(1+O(\d_t)+O(t \d_t^d)) .
\end{eqnarray*}
Hence for $k_1 < k_2$, asymptotically the covariance vanishes,
\begin{eqnarray}\label{eq:error-cov-sparse1}
\Cov\big(\widehat{\cV}_{k_1}^{(\a_1)},\widehat{\cV}_{k_2}^{(\a_2)}\big)
=
O((t\d_t^d)^{\frac {1}2 } ) .
\end{eqnarray}
In the case $k_1=k_2=k$, the covariance equals asymptotically a moment of the volume,
\begin{eqnarray}\label{eq:error-cov-sparse2}
\Cov\big(\widehat{\cV}_{k}^{(\a_1)},\widehat{\cV}_{k}^{(\a_2)}\big)
&=& \nonumber
\frac{ \m_{k}^{(\a_1+ \a_2)} }{(k+1)!} (1+O(\d_t + t \d_t^d))
\\ &=&
\E \widehat{\cV}_k^{(\a_1+\a_2)} (1+O(\d_t + t \d_t^d)) .
\end{eqnarray}
In the dense regime $t \d_t^d \to \infty$ and hence \eqref{eq:Ql} shows that the maximum is attained for $m=1$ as soon as $t \d_t^d >1$. Thus
$
Q_{1} Q_2
=
t \d_t^{ \sum \a_i  k_i \,  }  (t \d_t^d)^{\sum k_i} 
$
and 
\begin{eqnarray}\label{eq:error-cov-dense}
\Cov\big(\widehat{\cV}_{k_1}^{(\a_1)},\widehat{\cV}_{k_2}^{(\a_2)}\big)
&=& \nonumber
\frac{1}{Q_{1}Q_{2}}\Cov\big(\cV_{k_1}^{(\a_1)},\cV_{k_2}^{(\a_2)}\big)
\\&=& \nonumber 
\frac{ \m_{k_1, k_2:1}^{(\a_1, \a_2)} }{k_1! k_2!} (1+O(\d_t + (t \d_t^d)^{-1})) 
\\&=& 
\frac{ \m_{k_1}^{(\a_1)}\m_{k_2}^{(\a_2)}}{k_1! k_2!}
(1+O(\d_t + (t \d_t^d)^{-1}))  .
\end{eqnarray}
Note that in this case the limiting covariance is the product of the suitable normalized expectations, 
\begin{equation*}\label{eq:cov-ai-dense2}
\Cov\big(\widehat{\cV}_{k_1}^{(\a_1)},\widehat{\cV}_{k_2}^{(\a_2)}\big)
=
(k_1+1)(k_2+1) t^{-1} \E \widehat{\cV}_{k_1}^{(\a_1)} \E \widehat{\cV}_{k_2}^{(\a_2)} (1+o(1)) . 
\end{equation*}
In the thermodynamic regime, $t \d_t^d$ tends to a constant $c\in \R$, hence all terms in the sum occurring in the covariance contribute in the same way. 
Accordingly to the sparse regime we obtain for $c <1$
$$
Q_{1} Q_2
=
t \, \d_t^{ \sum \a_i  k_i   }  (t \d_t^d)^{\sum \frac {k_i}2}  
$$
for $t$ sufficiently large, and
$$\Cov\big(\widehat{\cV}_{k_1}^{(\a_1)},\widehat{\cV}_{k_2}^{(\a_2)}\big)
=\!\!
\sum_{m=1}^{k_1+1} \frac{ \m_{k_1, k_2:m}^{(\a_1, \a_2)} }{m! (k_1-m+1)! (k_2-m+1)!}
c^{\frac{k_1+k_2}2 -m+1} 
(1+o(1)) .
$$
Analogously, for $c \geq 1$ we obtain 
\begin{align*}
\Cov\big(\widehat{\cV}_{k_1}^{(\a_1)},\widehat{\cV}_{k_2}^{(\a_2)}\big)
= &
\sum_{m=1}^{k_1+1} \frac{ \m_{k_1, k_2:m}^{(\a_1, \a_2)}
 }{m! (k_1-m+1)! (k_2-m+1)!}
c^{-m+1} 
(1+o(1)) .
\end{align*}
In both cases we see that the error term $o(1)$ is given by 
\begin{equation}\label{eq:error-cov-thermo}
O(\d_t) + O(c-t\d_t^d) . 
\end{equation}

\end{proof}

Putting things together we obtain the limiting covariance matrix of the random vector $(\widehat{\cV}_{k_1}^{(\a_1)},\ldots,\widehat{\cV}_{k_n}^{(\a_n)})$.
For this we call $(k_1, \a_1), \dots, (k_n, \a_n)$ an \emph{admissible sequence} if 
\begin{enumerate}[(i)]
 \item $0 \leq k_1 \leq \dots \leq k_n$, 
 \item the pairs $(k_1, \a_1)$, \dots, $(k_n, \a_n)$ are distinct,
 \item $\a_i=0$ for $k_i >d$, and 
 \item $\min\{\a_i, \a_j, \a_i+\a_j\} > -d$ for all $i,j \in \{1, \dots, n\}$.
\end{enumerate}
First we rewrite the sum occurring in Theorem \ref{thm:covar-normalized} in the case $c<1$.
\begin{align*}
\sum\limits_{m=0}^{k_1} 
& 
\frac{ \m_{k_1, k_2:k_1-m+1}^{(\a_1, \a_2)} }{(k_1-m+1)! m! (k_2-k_1+m)!} c^{\frac{k_2-k_1}2 +m}
\\ = &
\sum_{m=0}^{\infty} 
\m_{k_l, k_1:\frac{(k_2+k_1-m+2)}2}^{(\a_1, \a_2)}
\frac{ \1(m -(k_2-k_1)\in \{0, 2,4, \dots, 2 k_1 \}) }{(\frac {k_2+k_1-m+2}2)! (\frac {m-k_2+k_1}2)! (\frac {m+k_2-k_1}2)!} 
\, c^{\frac m2}   \, .
\end{align*}
We define for $m=0, \dots, 2k_n$ the $(n \times n)$-matrices
\begin{equation}\label{def:A<1}
A^{<1}_m
=
\left(
\m_{k_l, k_j:\frac{(k_l+k_j-m+2)}2}^{(\a_1, \a_2)}
\frac{ \1(m -|k_l-k_j|\in \{0, 2,4, \dots, 2 \min k_i \}) }{(\frac {k_l+k_j-m+2}2)! (\frac {m-k_l+k_j}2)! (\frac {m+k_l-k_j}2)!} 
\right)_{l,j=1, \dots, n}
\end{equation}
with $\min k_i$ short for $\min_{i \in \{j,l\}} k_i$, and for $m=0, \dots , k_n$
\begin{equation}\label{def:A>1}
A^{>1}_m
=
\left(
\m_{k_l, k_j:m+1}^{(\a_l, \a_j)} \frac{\1(m \leq \min k_i)}{(m+1)! (k_l-m)!(k_j-m)!}\ 
\right)_{l,j=1, \dots, n} .
\end{equation}
Note that for large $m$ these matrices contain a large number of zeros. E.g., for $k_{n-1}<m \leq k_n$ the matrix $A^{>1}_{m}$  contains only one nonzero entry,
$$(A^{>1}_m)_{nn}=
\frac{ \m_{k_n, k_n:m+1}^{(\a_n, \a_n)} }{(m+1)! ((k_n-m)!)^2} .
$$

\medskip
In the following theorem we assume that $(k_1, \a_1)$, \dots, $(k_n, \a_n)$ is an admissible sequence.
Define the normalized volume-power functionals by $\widehat{\cV}_{k_i}^{(\a_i)}= \cV_{k_i}^{(\a_i)}/Q_{i}$
 with $Q_i $ defined in \eqref{def:Qi}.

\begin{theorem}\label{thm:covar-matrix}
The random vector $(\widehat{\cV}_{k_1}^{(\a_1)},\ldots,\widehat{\cV}_{k_n}^{(\a_n)})$ has the asymptotic covariance matrix 
\begin{equation}\label{eq:lim-cov-matrix}
\S =:
\left\{ 
\begin{array}{ll}
A^{<1}_0 & :\lim\limits_{t\rightarrow\infty}t\d_t^d=0 \\ 
\sum\limits_{m=0}^{2k_n}A^{<1}_m c^{\frac m2}  & :\lim\limits_{t\rightarrow\infty}t\d_t^d=c\in(0,1] \\ 
\sum\limits_{m=0}^{k_n}A^{>1}_m c^{-m} & :\lim\limits_{t\rightarrow\infty}t\d_t^d=c\in [1,\infty) \\
A^{>1}_0 & :\lim\limits_{t\rightarrow\infty}t\d_t^d=\infty 
\end{array}
\right. 
\end{equation}
\end{theorem}

Clearly, in the case $c=1$ the identity 
$\sum A^{<1}_m = \sum A^{>1}_m$
is satisfied which follows from the definitions \eqref{def:A<1} and \eqref{def:A>1}.

By Theorem \ref{thm:covar-normalized}, \eqref{eq:cov-ai-dense1}, the matrix $A^{>1}_0$ takes the form of a tensor product.
$$
A^{>1}_0 = \begin{pmatrix} \vdots\\ \frac{ \m_{k_i}^{(\a_i)} }{k_i!} \\ \vdots \end{pmatrix} \! \otimes  \!
\begin{pmatrix} \vdots\\ \frac{ \m_{k_i}^{(\a_i)} }{k_i!} \\ \vdots \end{pmatrix} 
 $$
Hence, in the dense case the covariance matrix $\S$ is of rank 1, and thus is singular in this regime. 

Also, the covariance matrix $A^{<1}_0 $ takes a particular nice form. Using \eqref{eq:error-cov-sparse2} we see that
\begin{equation*}
A^{<1}_0
=
\left(
\m_{k_j}^{(\a_1+ \a_2)}
\frac{ \1( k_l=k_j) }{(k_j+1)! } 
\right)_{l,j=1, \dots, n}
\end{equation*}
is a diagonal block matrix. A block is of size $i$ if $k_{m}=\dots=k_{m+i-1}$, and then is a constant times the matrix
$$
\left( \E \D_1^{(\a_l+\a_j)} \right)_{l,j=m, \dots , m+i-1}
$$
with
$ \D_1 = \D_1 [0, \{X_l\}_{l=1}^{k_m} ]$. 
Thus each block is a generalized moment matrix, and we known by Theorem~\ref{thm:rank-moment-matrix} that this is of full rank if $\a_l \neq \a_j$ for $l \neq j$.
Since all $\a_i$ are distinct, $A^{<1}_0$ is of full rank.

Further, on $c \in [0,1]$ the determinant $|\S|$ of the covariance matrix $\S$ is a polynomial in $c$ with $\lim_{c \to 0} |\S| \to |A^{<1}_0| >0$ and thus this polynomial is not trivial. Hence it has at most finitely many zeros. Analogously, for $c \in [1, \infty)$ the determinant of $\S$ is a polynomial in $c$. Because $\widehat \cV_{k_i}^{\a_i}$ on $c \geq 1$ is just a renormalized version of $\widehat \cV_{k_i}^{\a_i}$ on $c \leq 1$ the polynomial is not trivial. (In the limit the renormalizations are just multiplications by $c^{\frac{k_i}2}$.) Hence there are again only finitely many zeros of this polynomial.  
We summarize our findings.

\begin{corollary}\label{cor:rank}
The rank of $\S$ equals $n$ in the sparse regime. In the thermodynamic regime $\S$ is of rank $n$ except for finitely many values of $c$. In the dense regime $\S$ is of rank one.
\end{corollary}

\subsection{The face numbers of the simplicial complex}

If we are interested only in the combinatorial structure of the simplicial complex we have to consider in particular the case $\a=\a_i=0$. In this case, $\cV^{(0)}_k=f_k(\cR(\eta_t,\d_t))$ are just the entries of the  $f$-vector of the simplicial complex.
The constants simplify slightly, giving
\begin{equation*}
\m_{k}^{(0)} = \k_d^k\  \P (\D_{1}[0, \{ X_l\}_{l=1}^k]\neq 0)
\end{equation*}
which just is the probability that $k$ random points in $B^d$ have mutual distances at most one, 
and, analogously,
\begin{align*}\label{def:C00kkm}
\m_{k_1, k_2:m}^{(0,0)}
=& 
\k_d^{k_1+k_2+1-m}
\P (\D_{1} [0,\{x_l\}_{l=1}^{k_1}] \neq 0,\, 
\D_{1} [0, \{x_l\}_{l=k_1-m+2}^{k_1+k_2-m+1}]\neq 0) .
\end{align*}

The $\bm f$-vector $(f_k)_{k \geq 0}= (f_k(\cR(\eta_t,\d_t)))_{k \geq 0}$ of the Vietoris-Rips complex satisfies 
$$\E f_k 
=
\frac{\m_k^{(0)}}{(k+1)!} \, t(t\d_t^d)^{k} (1+O(\d_t)),$$
and similar formulae hold for the covariance matrix of $\widehat f_k$ by Theorem \ref{thm:covar-matrix}.

\section{Central Limit Theorems}

\subsection{Some Results for Poisson U-statistics}

A Poisson U-statistic is absolutely convergent if $F= \sum _{\eta_{t,\neq}^k} |f(x_1, \dots, x_k)|$ is in $L^2(\P)$.  
Note that $\widehat{\cV}_k^{\a}$ is an absolutely convergent U-statistic since all occurring functions are bounded and vanish outside the compact convex set $W$. 
Let $F^{(1)},\ldots,F^{(n)}$ be absolutely convergent Poisson U-statistics of order $k_1,\ldots,k_n$ respectively, 
$$
F^{(l)}=\sum\limits_{(x^{(l)}_1,\ldots,x^{(l)}_{k_l})\in\eta_{t,\neq}^{(k_l)}}f^{(l)}(x^{(l)}_1,\ldots,x^{(l)}_{k_l})
$$ 
for $l=1,\ldots,n$. It will be essential to define suitable partitions on the set of variables
$ \{x^{(l)}_1,\ldots,x^{(l)}_{k_l}\}$, $l=1,\ldots,n$,  
of $f^{(l)}:W^{k_l}\rightarrow\bar{\R}$. 

Let $\cP(A)$ stand for the set of partitions of an arbitrary set $A$. Then $|\s|$ represents the number of blocks in a partition, $\s\in\cP(A)$. A partial order is defined on $\cP(A)$ such that $\s\leq\tau$ if each block of $\s$ is contained in a block of $\tau$, for $\s,\tau\in\cP(A)$. The minimal partition $\hat{0}$ is the partition whose blocks are singletons, and the maximal partition $\hat{1}$ is the partition with a single block. For two partitions $\s,\tau\in\cP(A)$, $\s\wedge\tau$ is the maximal partition in $\cP(A)$ such that $\s\wedge\tau\leq\s$ and $\s\wedge\tau\leq\tau$, and $\s\vee\tau$ is the minimal partition in $\cP(A)$ such that $\s \leq \s\vee\tau$ and $\tau \leq \s\vee\tau$. 

Set $V(k_1,\ldots,k_4)=\Big\{x^{(1)}_1,\ldots,x^{(1)}_{k_1},x^{(2)}_{1},\ldots,x^{(3)}_{k_{3}},x^{(4)}_1,\ldots,x^{(4)}_{k_4}\Big\}$, which consists of four sets of variables, and let $\bar{\pi}\in\cP(V(k_1,\ldots,k_4))$ be the partition whose blocks are the fundamental building blocks $\{x^{(l)}_1,\ldots,x^{(l)}_{k_l}\}$, $l=1,\ldots,4$. 

The set 
$$
\tilde \Pi(k_1,\ldots,k_4)
=
\Big\{\s\in\cP(V(k_1,\ldots,k_4))\colon \s\wedge\bar{\pi}=\hat{0},\, \s\vee\bar{\pi}=\hat{1} \Big\}
$$
is the set of all partitions such that each block contains at most one element from each of the building blocks $\{x^{(l)}_1,\ldots,x^{(l)}_{k_l}\}$, $l=1,\ldots,4$, and all four fundamental blocks are connected. Clearly, for $\s \in \tilde \Pi(k_1, \dots, k_4)$  it may happen that some variables are singletons, and we define $s(\s)=(s_1,\dots, s_4)$ to be the vector consisting of the number of singletons in each of the building blocks. 

In the following we need the notion of a 4-fold tensor product, $\otimes_{l=1}^4 f^{(l)}:W^{\sum_{l=1}^4 k_l}\rightarrow \R$, of functions $f^{(l)}$, given by
$$
\left( \otimes_{l=1}^4 f^{(l)} \right) (x^{(1)}_1,\ldots,x^{(4)}_{k_4})
=
\prod_{l=1}^{4} f^{(l)}(x^{(l)}_1,\ldots,x^{(l)}_{k_l}).
$$
For a partition, $\s\in\Pi(k_1,\ldots,k_4)$, we construct a new function, $(\otimes_{l=1}^mf^{(l)})_{\s}:W^{|\s|}\rightarrow\bar{\R}$, by replacing all variables that belong to the same block of $\s$ by a new common variable. We refer to \cite{S13} for more details and some examples. 
Finally we are able to introduce the functions $M_{ij}$ defined in Reitzner and Schulte \cite{RS13} in the univariate case and by Schulte \cite{S13} for the multivariate case. The functions are given by\footnote{Remark that by Fubinis theorem one can integrate first the functions $f^{(l)}$  over the $k_l-i$ free variables, i.e. singletons, which produces reduced functions $f_i^{(l)}$, and analogously the functions $f^{(m)}$  over the $k_m-j$ free variables producing $f_j^{(m)}$. In this form the result was stated in \cite{S13}.}
$$
M_{ij}(f^{(l)},f^{(m)})
=
\sum\limits_{\stackrel{\s\in\tilde{\Pi}(k_l,k_l,k_m,k_m)}{s(\s)=(k_l -i, k_l -i, k_m -j, k_m -j)}}
\int\limits_{W^{|\s|}}
|( f^{(l)} \otimes f^{(l)} \otimes f^{(m)} \otimes f^{(m)})_{\s}|\ d\mu^{|\s|}
$$
where in our case $d\mu$ is the intensity measure $t dx$. Apart from the precise definition given above the main point is that the functions $M_{ij}$ is something like a mixed fourth moment of $f^{(l)}$ and $f^{(m)}$ where all functions are linked via the common use of some of the variables. 
These functions $M_{ij}$ are the main ingredients in the following two central limit theorem.

\medskip
The univariate central limit theorem uses the Kolmogorov distance $d_K$ and the Wasserstein distance $d_W$ of random variables. For the Wasserstein distance the following theorem was proved  in Reitzner and Schulte \cite {RS13} using the Wiener-It\^{o} chaos expansion, and extended to the Kolmogorov distance by Schulte \cite{S16}.

\begin{theorem}\label{thm:uni-clt}
Let $F=F(\eta_t) \in L^2(\P)$ be an absolutely convergent U-statistic of order $k$ with $\Var F >0$, and $N$ be a standard Gaussian random variable. Then for $d_{\star}= d_W$ or $d_{\star}=d_K$ there is a constant $c_k$ such that 
$$
d_{\star}\bigg(\frac{F-\E F}{\sqrt{\Var\,F}}, N\bigg)\leq 
c_k \sum_{i,j=1}^k  \frac{\sqrt{M_{ij}(f,f)}}{\Var F} \ . 
$$
\end{theorem}

The mulitvariate central limit theorem makes use of the $d_3$-distance, which is obtained by taking ${\cal C}^3_{1}$ to be the set of functions $g:\R^n\rightarrow\R$ that are three times differentiable and all partial derivatives of order 2 and 3 are bounded by 1.

\begin{definition}\label{def:dist2}
Let $F,G$ be two $n$-dimensional random vectors. 
The $d_3$ distance is given by
$$d_3(F,G)=\sup_{g\in\mathcal{C}^3_{1}} |\E g(F)-\E g(G)|
. $$
\end{definition}
Next we state the central limit theorem proven by Schulte \cite{S13} which will be useful in computing the $d_3$-distance in the next section.
\begin{theorem}\label{thm:mlt}
Let $\bm{F}=(F^{(1)},\ldots,F^{(n)})$ be a vector of absolutely convergent Poisson $U$-statistics of orders $k_1,\dots,k_n$,  
$$F^{(l)}=\sum_{(x_1,\ldots,x_{k_l})\in\eta_{t,\neq}^{k_l}} f^{(l)}(x_1,\ldots,x_{k_l})  . $$
And let $\bm{N}(\Sigma)$ be an $n$-dimensional centered Gaussian random vector with a positive semidefinite covariance matrix $\Sigma$. Then 
\begin{align*}
d_3\bigg(\bm{F}-\E\bm{F},\bm{N}(\Sigma)\bigg)
&\leq
\frac{1}{2}\sum\limits_{l,m=1}^{n}|\sigma_{lm}-\Cov (F^{(l)},F^{(m)})|
\\&+
\frac{n}{2}\bigg(\sum\limits_{l=1}^{n}\sqrt{\Var F^{(l)}}+1\bigg)
\sum\limits_{l,m=1}^{n} \sum\limits_{i=1}^{k_l}\sum\limits_{j=1}^{k_m} 
k_l^{\frac{7}{2}}\sqrt{M_{ij}(f^{(l)},f^{(m)})}.
\end{align*}
\end{theorem}
We shall bound the terms on the right hand side in the next section.

\subsection{Central Limit Theorems for $\cV_k^{\a}$}
In the following theorem we assume that $\a  > - \frac d2$ for $0 \leq k \leq d$ and $\a=0$ for $k>d$. 

\begin{theorem}\label{thm:uni-clt-V}
Let $N$ be a standard Gaussian random variable. Then for $d_{\star}= d_W$ or $d_{\star}=d_K$ there is a constant $c_{k,\a}$ such that 
$$
d_{\star} \bigg(\frac{{\cV}_k^{\a}-\E {\cV}_k^{\a}}{\sqrt{\Var\,{\cV}_k^{\a}}}, N\bigg)
\leq 
c_{k,\a}  t^{- \frac 12} \max\{( t \d_t^d)^{- \frac k2 } , 1\} . 
$$
\end{theorem}

Note that it was to be expected that a central limit theorem only holds if $\E f_k \to \infty$ which happens if $t (t \d_t^d)^k \to \infty$. It turns out that this is precisely the requirement in Theorem \ref{thm:uni-clt-V}.

In the case $\a=0$ Theorem \ref{thm:uni-clt-V} just gives a univariate central limit theorem for the number of facets. For the Kolmogorov distance this is already well known due to work by Penrose \cite{P03} although the central limit theorems there come without error term. In a recent paper by Lachi\'eze-Rey, Schulte and Yukich \cite{LSY19} the error terms for the thermodynamic regime and the dense regime have been obtained as a consequence of a much more general theorem for stabilizing functionals. For the Wasserstein distance a central limit theorem with error bounds is due to Decreusefond et al. \cite{DFRV_14}.

\begin{proof}
We apply Theorem \ref{thm:uni-clt}. 
${\cV}_k^{(\a)}$ is an absolutely convergent Poisson U-statistic of order $k+1$ with
$$ 
 f(x_0, \dots, x_{k})= 
\frac{1}{(k+1)!} \D_{\d_t} [x_0, \dots, x_{k}]^{\a}  .
 $$
We have to show that the functionals $M_{ij}$ tend to zero. In our case the summands in $M_{ij}$ take the form
$$
t^{|\s |} \int\limits_{W^{|\s|}} 
\left|
( \D_{\d_t}[\cdot ]^{\a} \otimes \D_{\d_t}[ \cdot]^{\a} \otimes \D_{\d_t}[\cdot]^{\a}  \otimes \D_{\d_t}[ \cdot ]^{\a}  )_{\s}
\right| 
\ dx_0 \dots dx_{|\s|-1}
 $$
where the functionals depend on simplices of dimension $k$. 
The essential feature in the definition of $\s$ is that all four functionals  $\D_{\d_t}[ \cdot]  $ are linked by common variables, and each of these functionals depends on $k+1$ variables. First, for the number $|\s|$ of variables  this implies 
\begin{equation}\label{eq:boundsigma}
k +1 \leq |\s| \leq 4k +1.
\end{equation}
Second, assuming w.l.o.g. that $x_0$ occurs in the first functional, e.g. $\D_{\d_t}[ \cdot]=\D_{\d_t}[\{ \cdot\}_0^{k}]$, all other variables in this first function are at most at distance $\d_t$, in the functional directly linked to the first one by at most $2 \d_t$, etc. Thus
$$ \max \| x_i-x_0\| \leq 4 \d_t $$
if the integrand is not vanishing.
Further, $ \D_{\d_t}^{\a} \leq \k_{k}^{\a} \d_t^{k \a} . $ 
Hence 
\begin{align*}
t^{|\s |} 
&
\int\limits_{W^{|\s|}} 
\left| ( \D_{\d_t}[\cdot ]^{\a_l} \otimes \D_{\d_t}[ \cdot]^{\a_l} \otimes \D_{\d_t}[\cdot]^{\a_m}  \otimes \D_{\d_t}[ \cdot ]^{\a_m}  )_{\s}
\right|
\ dx_0 \dots dx_{|\s|-1}
\\ & \leq 
c_{k,\a} t^{|\s |} (\d_t)^{4 k \a }
\int\limits_{W^{|\s|}} 
\1( \forall i\colon \| x_i - x_0 \| \leq 4\d_t)
\ dx_0 \dots dx_{|\s|-1}
\\ & \leq 
c_{k,\a} t^{|\s |} (\d_t)^{4 k \a }
(4\d_t)^{d (|\s|-1)}  .
\end{align*}
By \eqref{eq:boundsigma} this implies
\begin{equation*}
M_{ij} ( \D_{\d_t}[\cdot]^{\a}, \D_{\d_t}[\cdot]^{\a} )
\leq 
c_{k,\a} t \, \d_t^{4 k \a} \max \{ (t \d_t^d)^{k},  (t \d_t^d)^{4k} \}
.  \end{equation*}
Next we use the variance asymptotics from Theorem \ref{thm:covar}. They imply 
\begin{eqnarray*}
\Var\,\cV_{k}^{(\a)}
& \geq &
c_{k,\a} t \d_t^{2\a k  }  
\max\{  (t\d_t^d)^{k } , (t \d_t^d)^{2 k  }  \}  
\end{eqnarray*}
for $\d_t$ sufficiently small. 
This shows
\begin{align*}
\frac{\sqrt{M_{ij} ( \D_{\d_t}[\cdot]^{\a}, \D_{\d_t}[\cdot]^{\a} )} }{\Var \,\cV_{k}^{(\a)} }
& \leq 
c_{k,\a} t^{- \frac 12} 
\max \{ (t \d_t^d)^{ - \frac k2}, 1\} 
.  \end{align*}
Summing over all $M_{ij}$  gives the desired result.
\end{proof}

\subsection{Multivariate Central Limit Theorem}
In the following theorem we assume that $(k_1, \a_1)$, \dots, $(k_n, \a_n)$ is an admissible sequence.
Define the normalized volume-power functionals by 
$\widehat{\cV}_{k_i}^{(\a_i)}= \cV_{k_i}^{(\a_i)}/Q_{i}$ 
with $Q_i $ defined in \eqref{def:Qi}.
 \begin{theorem}\label{thm:mult-clt-V}
 Let $\bm{V_k^{(\a)}}=(\widehat{\cV}_{k_1}^{(\a)},\ldots,\widehat{\cV}_{k_n}^{(\a)})$, and let $\bm{N}(\Sigma_t)$ be the $n$-dimensional centered Gaussian random vector with covariance matrix 
 $$
 \Sigma_t=(\s_{l,m})_{lm}
\mbox{ \ with \ }
\s_{l,m} =
\Cov (\widehat{\cV}_{k_l}^{(\a_l)},\widehat{\cV}_{k_m}^{(\a_m)})
 .$$
 Then there is a constant $c_{\bm k, \bm \a} $ such that 
\begin{eqnarray*} 
d_3\bigg(\bm{V_k^{(\a)}}-\E\bm{V_k^{(\a)}},\bm{N}(\Sigma_t)\bigg)
&\leq&
c_{\bm k, \bm \a}  t^{- \frac 12} \,  \max \{ 1, (t \d_t^d)^{-\frac 12 k_n} \} .
\end{eqnarray*}
 \end{theorem}

Thus in the dense and thermodynamic regime a central limit theorem holds with rate of convergence $t^{-\frac12}$ which most probably is optimal.
In the sparse regime, where $t \d_t^d \to 0$, the rate of convergence is 
$$ t^{- \frac 12} \, (t \d_t^d)^{-\frac 12 k_n} = \Theta( (\E f_{k_n}(\cR(\eta_t,\d_t)))^{- \frac 12}) , $$
and thus there is a multivariate central limit theorem as long as the expectation $\E f_k$ tend to infinity for all $k \in \bm k$. To the best of our knowledge Theorem \ref{thm:mult-clt-V} is new even in the case $\a_i=0$ where we obtain a central limit theorem for the $\bf f$-vector of the Vietoris-Rips complex and for the \v{C}ech complex.

In the view of the first term on the RHS in Theorem~\ref{thm:mlt} it is of interest to state the difference between $\S_t$ and the limiting covariance matrix $\S$ given in \eqref{eq:lim-cov-matrix}. By equations \eqref{eq:error-cov-sparse1}, \eqref{eq:error-cov-sparse2}, \eqref{eq:error-cov-dense}, and \eqref{eq:error-cov-thermo} we see that in the sparse case
$$
\frac{1}{2}\sum\limits_{l,m=1}^{n}|\sigma_{lm}-\Cov (\widehat{\cV}_{k_l}^{(\a_l)},\widehat{\cV}_{k_m}^{(\a_m)})|
\leq O(\d_t + t \d_t^d), 
$$
that in the thermodynamic regime this error term is of order
$$ O(\d_t +(c- t\d_t^d)) , $$
and in the dense regime of order
$$ O(\d_t + (t \d_t^d)^{-1}) . $$
Thus the $d_3$-distance 
$ d_3 (\bm{V_k^{(\a)}}-\E\bm{V_k^{(\a)}},\bm{N}(\Sigma) ) $
would have this additional error terms.

\begin{proof}
By definition the first term on the RHS in Theorem~\ref{thm:mlt} vanishes. And by Theorem \ref{thm:covar-normalized} the variance $\Var \widehat{\cV}_{k_l}^{(\a_l)}$ tends to a constant. Hence we just have to show that the functionals $M_{ij}$ tend to zero. In our case the summands in $M_{ij}$ take the form
$$
t^{|\s |} \int\limits_{W^{|\s|}} 
\frac{\left|
( \D_{\d_t}[\cdot ]^{\a_l} \otimes \D_{\d_t}[ \cdot]^{\a_l} \otimes \D_{\d_t}[\cdot]^{\a_m}  \otimes \D_{\d_t}[ \cdot ]^{\a_m}  )_{\s}
\right|}  {Q_l^2 Q_m^2}
\ dx_0 \dots dx_{|\s|-1}
 $$
where the first two functionals depend on simplices of volume $k_l$ and the other two on simplices of dimension $k_m$. Assume from now on that $k_l \leq k_m$. 
 The essential feature in the definition of $\s$ is that all four functionals  $\D_{\d_t}[ \cdot]  $ are linked by common variables, and each of these functionals depends on $k_l+1$, resp. $k_m+1$ variables. First, for the number $|\s|$ of variables  this implies 
 $$
\max \{ k_l, k_m\} +1 =k_m +1\leq |\s| \leq 2(k_l+k_m) +1
 $$
 since $k_l \leq k_m$. Second, assuming w.l.o.g. that $x_0$ occurs in the first functional, $\D_{\d_t}[ \cdot]=\D_{\d_t}[\{ \cdot\}_0^{k_l}]$, all other variables in this first function are at most at distance $\d_t$, in the functional directly linked to the first one by at most $2 \d_t$, etc. Thus
$$ \max \| x_i-x_0\| \leq 4 \d_t $$
if the integrand is not vanishing.
Further, 
$$ \D_{\d_t}^{\a_l} \leq \k_{k_l}^{\a_l} \d_t^{k_l \a_l}
\ \mbox{ and  }\ 
\D_{\d_t}^{\a_m} \leq \k_{k_m} \d_t^{k_m \a_m}  
. $$ 
Hence 
\begin{align*}
t^{|\s |} 
&
\int\limits_{W^{|\s|}} 
\frac{\left|
( \D_{\d_t}[\cdot ]^{\a_l} \otimes \D_{\d_t}[ \cdot]^{\a_l} \otimes \D_{\d_t}[\cdot]^{\a_m}  \otimes \D_{\d_t}[ \cdot ]^{\a_m}  )_{\s}
\right|}  {Q_l^2 Q_m^2}
\ dx_0 \dots dx_{|\s|-1}
\\ & \leq 
c_{k_l, \a_l, k_m, \a_l} t^{|\s |} 
\frac{ \d_t^{2(k_l \a_l  + k_m \a_m)} }  {Q_l^2 Q_m^2} \, 
(4 \d_t)^{d(|\s|-1)} 
% \frac{ \d_t^{2(k_l \a_l  + k_m \a_m)} }  {Q_l^2 Q_m^2} \,  (t \d_t^d)^{|\s|-1} 
\\ & \leq 
c_{k_l, \a_l, k_m, \a_l} t \, 
\frac{ \d_t^{2(k_l \a_l  + k_m \a_m)} }  {Q_l^2 Q_m^2} \, 
\max \{ (t \d_t^d)^{k_m},  (t \d_t^d)^{2(k_l+k_m)} \}
\end{align*}
which implies
\begin{equation*}
M_{ij} \left( \frac{\D_{\d_t}[\cdot]^{\a_l}}{Q_l}, \frac{\D_{\d_t}[\cdot]^{\a_m}}{Q_m} \right)
\leq 
c_{\bm k, \bm \a} t \, 
\frac{ \d_t^{2(k_l \a_l  + k_m \a_m)} }  {Q_l^2 Q_m^2} \, 
\max \{ (t \d_t^d)^{k_m},  (t \d_t^d)^{2(k_l+k_m)} \}
.  \end{equation*}
Plugging the definition \eqref{def:Qi} of $Q_i$ into this shows
\begin{align*}
M_{ij} \left( \frac{\D_{\d_t}[\cdot]^{\a_l}}{Q_l}, \frac{\D_{\d_t}[\cdot]^{\a_m}}{Q_m} \right)
& \leq 
%c_{k_l, k_m} t \, 
%\frac{ \d_t^{2(k_l \a_l  + k_m \a_m)} \max \{ (t \d_t^d)^{k_m},  (t \d_t^d)^{2(k_l+k_m)} \}}  
%{ t\d_t^{2 \a_l k_l} \max \{  (t \d_t^d)^{ 2 k_l}, (t \d_t^d)^{k_l }  \}  t\d_t^{2 \a_i k_i} \max \{  (t \d_t^d)^{2 k_m} ,(t \d_t^d)^{k_m }  \}  } \, 
c_{\bm k, \bm \a} t^{-1} \, 
\frac{ \max \{ (t \d_t^d)^{k_m},  (t \d_t^d)^{2(k_l+k_m)} \}}  
{ \max \{ (t \d_t^d)^{k_l +k_m}   , (t \d_t^d)^{2(k_l+k_m )}  \}  } \, 
\\ & =
c_{\bm k, \bm \a}   t^{-1}  \max \{ 1, (t \d_t^d)^{-k_l} \}   
.  \end{align*}
The sum over all $M_{ij}$ yields
\begin{eqnarray*}
\sum\limits_{l,m=1}^{n} \sum\limits_{i=1}^{k_l}\sum\limits_{j=1}^{k_m} 
k_l^{\frac{7}{2}}\sqrt{M_{ij}(f^{(l)},f^{(m)})}
 \leq 
c_{\bm k, \bm \a} t^{- \frac 12} \,  \max \{ 1, (t \d_t^d)^{-\frac 12 k_n} \} 
.  
\end{eqnarray*}

\end{proof}

\section{The \v{C}ech complex}\label{sec:Cech}
It is immediate from the definition of the Vietoris-Rips complex and the \v{C}ech complex that 
$$\cC(\eta_t, \d_t) \subset \cR (\eta_t, \d_t) \subset \cC(\eta_t, (\tfrac {2d}{d+1})^{\frac 12}\  \d_t) .
$$
Hence all bounds obtained for the Vietoris-Rips complex hold true for the \v{C}ech complex with constants changed by a factor of $(\tfrac {2d}{d+1})^{\frac 12}$. The constants in the expectation and covariance change in the follwing way. Denote by $\D^c_{s} [x_0, \dots x_k]$ the $k$-dimensional volume of the convex hull of the points $x_0, \dots, x_k$ if the intersection $\bigcap_1^k B^d (x_i, \tfrac s2 )$ is not empty, and set $\D^c_{s} [x_0, \dots, x_k]=0$ otherwise. In the case $k >d$ we only define $\D^c_{s} [x_0, \dots x_k]^0=1 $ if the intersection property holds. Thus for all $k \geq 0$,
$$ 
F \in \cC_k (\cR (\eta_t, \d_t)) \Leftrightarrow 
\D^c_{\d_t}(F)^0=1  .
$$
We define 
\begin{equation*}
\n_{k}^{(\a)} = 
\int\limits_{(B^d)^k}\D^c_{1} [0,\{x_l\}_{l=1}^k]^{\a}  \ \dm x_1\cdots\dm x_k
\end{equation*}
with $\n_0^ {(\a)}=1$. In the case $k >d$ the definition only applies to $\a=0$. Again $\a>-d$ ensures that $\n_k^{(\a)} < \infty$.
The volume-power functional of the \v{C}ech complex is given by
$$
\cU_k^{(\a)}
=
\frac{1}{(k+1)!}\sum_{(x_0, \dots, x_k) \in \eta_{t, \neq}^k} \D^c_{\d_t}[x_0, \dots, x_k]^{\a}  .
$$
Then the \v{C}ech complex version of Theorem \ref{thm:expt} holds for $\cU_k^{(\a)}$  with $\m_k^{(\a)} $ replaced by $\n_k^{(\a)} $.
Analogously, define for $k_1 \leq k_2$
\begin{align*}
\n_{k_1, k_2:m}^{(\a_1, \a_2)}
=& 
\int\limits_{(B^d)^{k_1+k_2+1-m}}
\D^c_{1} [0,\{x_l\}_{l=1}^{k_1}]^{\a_1} 
\D^c_{1} [0, \{x_l\}_{l=k_1-m+2}^{k_1+k_2-m+1}]^{\a_2}
\\ & \hskip6cm  \nonumber
\dm x_1\cdots \dm x_{k_1+k_2-m+1} 
\end{align*}
and by symmetry $\n_{k_1, k_2:m}^{(\a_1, \a_2)} = \n_{k_2, k_1:m}^{(\a_2, \a_1)} $. Then Theorem \ref{thm:covar}, Theorem \ref{thm:covar-normalized} and Theorem \ref{thm:covar-matrix} hold for $\cU_{k_i}^{(\a_i)}$ with a covariance matrix $\S^c$  with $\m_{k_1, k_2:m}^{(\a_1, \a_2)}$ replaced by $\n_{k_1, k_2:m}^{(\a_1, \a_2)}$.
Finally, the proofs of the central limit theorems only depend on the local behavior of the random simplicial complexes and thus the identical proof holds for the \v{C}ech complex. Define the normalized volume-power functionals by 
$$\widehat{\cU}_{k_i}^{(\a_i)}=\frac{1}{Q_{i}} \cU_{k_i}^{(\a_i)} $$ 
with $Q_i $ defined in \eqref{def:Qi}, and let $\bm{U_k^{(\a)}}=(\widehat{\cU}_{k_1}^{(\a)},\ldots,\widehat{\cU}_{k_n}^{(\a)})$.
Assume that $(k_1, \a_1)$, \dots, $(k_n, \a_n)$ is an admissible sequence.
\begin{theorem}\label{thm:uni-mult-clt-C}
For $d_{\star}= d_W$ or $d_{\star}=d_K$ there is a constant $c_k$ such that 
$$
d_{\star} \bigg(\frac{{\cU}_k^{\a}-\E {\cU}_k^{\a}}{\sqrt{\Var\,{\cU}_k^{\a}}}, N\bigg)
\leq 
c_k  t^{- \frac 12} \max\{( t \d_t^d)^{- \frac k2 } , 1\} . 
$$
And there is a constant $c_{\bm k, \bm \a}$ such that 
\begin{eqnarray*} 
d_3\bigg(\bm{U_k^{(\a)}}-\E\bm{U_k^{(\a)}},\bm{N}(\Sigma^c_t)\bigg)
&\leq&
c_{\bm k,\bm \a} t^{- \frac 12} \,  \max \{ 1, (t \d_t^d)^{-\frac 12 k_n} \} 
\end{eqnarray*}
where $\Sigma^c_t$ is the covariance matrix 
$
 \S^c_t=( \Cov (\widehat{\cU}_{k_l}^{(\a_l)},\widehat{\cU}_{k_m}^{(\a_m)}) )_{lm}
$.
 \end{theorem}

\end{document}